\documentclass[12pt,a4paper]{article}
\usepackage[active]{srcltx}
\usepackage{amsfonts}
\usepackage{amsmath}
\usepackage{amsbsy}
\usepackage{amsxtra}
\usepackage{latexsym}
\usepackage{amssymb}
\usepackage{url}
\usepackage{cite}
\usepackage{pstricks,pst-node}
\makeatletter
\g@addto@macro\thesection.
\makeatother

\newtheorem{theorem}{Theorem}
\newtheorem{lemma}{Lemma}
\newtheorem{corollary}{Corollary}
\newtheorem{proposition}{Proposition}
\newtheorem{remark}{Remark}
\newtheorem{definition}{Definition}
\newtheorem{example}{Example}

\DeclareMathOperator{\intr}{int}
\DeclareMathOperator{\mmx}{max}

\DeclareMathOperator{\ddm}{dim}

\DeclareMathOperator{\argmin}{argmin}

\DeclareMathOperator{\im}{im}

\newcommand{\lang}{\langle}
\newcommand{\rang}{\rangle}

\newcommand{\R}{\mathbb R}

\newenvironment{proof}{{\noindent\bf Proof.}}{\hfill$\Box$\\}

\begin{document}

\date{}

\title{Subadditive retractions on cones and asymmetric vector norms}
\author{A. B. N\'emeth\\Faculty of Mathematics and Computer Science\\Babe\c s Bolyai University, Str. Kog\u alniceanu nr. 1-3\\RO-400084
Cluj-Napoca, Romania\\email: nemab@math.ubbcluj.ro \and S. Z. N\'emeth\\School of Mathematics, The University of Birmingham\\The Watson Building,
Edgbaston\\Birmingham B15 2TT, United Kingdom\\email: s.nemeth@bham.ac.uk}
\date{}
\maketitle
 
\maketitle

\begin{abstract}
Asymmetric vector norms are generalizations of asymmetric norms, where the subadditivity
inequality is understood in ordered vector space sense. This relation 
imposes strong conditions on the ordering itself. This note studies on these
conditions in the general case, and in the case when the asymmetric vector norm is the metric
projection onto the cone engendering the order relation.

\end{abstract}

\section{Introduction}

Let $X$ denote a real vector space.
The functional $q: X \rightarrow \R_+= [0,+\infty)$ is called an \emph{asymmetric norm}
if the following conditions hold:

\begin{definition}\label{assnor}
\begin{enumerate} 
\item $q(tx)=tq(x),\, \forall t\in \R_+,\,\,\forall x\in X$;
\item $q(x+y)\leq q(x)+q(y),\,\,\forall x,\,y \in X$;
\item If $q(x)=q(-x)=0$ then $x=0$. 
\end{enumerate}
\end{definition}

The couple $(X,q)$ will be called \emph{asymetrically normed space}. It is easy to see that item 1 in
Definition \ref{assnor} implies $q(0)=0$.

Asymmetrically normed spaces give rise to an extended literature (see e. g. the monograph of S. Cobzas \cite{Cobzas2013}
and the literature therein). The circumstance that the set
$C_q=\{x\in X: \,q(-x)=0\}$ is a cone, relates these spaces to the
ordered vector space theory. The recent paper  J Conradie \cite{Conradie2015} exploits
this relation in the context of topological vector spaces.
The relation with ordered vector space theory becomes tightened if $q$ is an operator and in the item 2
of Definition \ref{assnor} the inequality is an order relation in $X$.

This formal generalization and the fact that the positive part operator of a vector lattice
obviously satisfies the conditions in Definition \ref{assnor} leads to the theory
in vector lattices. But the definition of an asymmetric vector norm concerns general ordered
vector spaces. At this point an unexpected difficulty appears: 
the condition imposed on an asymmetric vector norm seems to impose hard conditions
on the ordering itself. A recent paper of M. Ilkan \cite{Ilkhan2020} considers
the general definition ignoring the question if such a asymmetric vector norm exists
or not in a not lattice ordered vector space.

This note studies the existence of asymmetric vector norms. We shall show that this approach opens a new 
unifying perspective on some earlier ordered vector space results raising some new exciting problems. 


\section{Preliminaries}\label{secprel}

The nonempty set $K\subset X$ is called a \emph{convex cone} if 
\begin{enumerate}
	\item[(i)] $\lambda x\in K$, for all $x\in K$ and $\lambda \in \R_+$ and if 
	\item[(ii)] $x+y\in K$, for all $x,y\in K$. 
\end{enumerate}

A convex cone $K$ is called \emph{pointed} if $K\cap(-K)=\{0\}$. 

A convex cone is called \emph{generating} if $K-K=X$. 

If $X$ is a topological vector space then a closed, pointed generating convex cone in it is called \emph{proper}.

The \emph{relation $\leq$ defined by the pointed convex cone $K$} is given by $x\leq y$ if and only if 
$y-x\in K$. Particularly, we have $K=\{x\in X:0\leq x\}$. The relation $\leq$ is an \emph{order relation}, that 
is, it is reflexive, transitive and antisymmetric; it is \emph{translation invariant}, that is, $x+z\leq y+z$, 
$\forall x,y,z\in X$ with $x\leq y$; and it is \emph{scale invariant}, that is, $\lambda x\leq\lambda y$, 
$\forall x,y\in X$ with $x\leq y$ and $\lambda \in \R_+$.

Conversely, for every $\leq$ scale invariant, translation invariant and antisymmetric order relation in $X$ there is a pointed 
convex cone $K$, defined by $K=\{x\in X:0\leq x\}$, such that $x\leq y$ if and only if $y-x\in K$. The cone $K$ 
is called the \emph{positive cone of $X$} and $(X,\leq)$ (or $(X,K)$) is called an \emph{ordered vector space}; we use
also the notation $\leq =\leq_K.$ 

The mapping $R:(X,\leq)\to (X,\leq)$ is said to be \emph{isotone} if $x\leq y\,\Rightarrow Rx\leq Ry.$
$R$ is called \emph{subadditive} if it holds $R(x+y)\leq Rx+Ry$, for any $x,\,y\in X$.

The ordered vector space $(X,\leq)$ is called \emph{lattice ordered} if for every  $x,y\in X$ there exists 
$x\vee y:=\sup\{x,y\}$. In this case the positive cone $K$ is called a \emph{lattice cone}.
A lattice ordered vector space is called a \emph{Riesz space}. 
Denote $x^+=0\vee x$ and $x^-=0\vee (-x)$. Then, $x=x^+-x^-$, $x^+$ is 
called the \emph{positive part} of $x$ and $x^-$ is called the \emph{negative part} of $x$. 
The \emph{absolute value} of $x$ is defined by $|x|=x^++x^-$. The mapping 
$x\mapsto x^+$ is called the \emph{positive part mapping}.

If $X$ is endowed with a
vector space topology, then the continuity of the positive
part mapping is equivalent to the closedness of $K$.

The closed, pointed cone $K\subset X$ with $X$ a Banach space is called \emph{normal} 
if from $x_m\in K$, $x_m\to 0$ and
$0\leq y_m\leq x_m$ it follows $y_m\to 0.$

Every closed pointed cone in a finite dimensional Banach space is normal.


In the particular case of $X=H$ with $(H,\lang,\rang)$ a separable Hilbert space (with $\lang,\rang$ the scalar product),
 we shall need some further notions.
Let $K\subseteq H$ be a closed convex cone. Recall that 
$$K^*=\{x\in H:\langle x,y\rangle\geq0,\,\forall y\in K\}$$
is called the \emph{dual cone} of $K$. $K^*$ is a closed convex cone and if $K$ is generating, then $K^*$ is 
pointed (this is the case for example if $K$ is latticial).

When $(H,\lang,\rang)=(\R^m,\lang.\rang)$, the $m$-dimensional Euclidean space
the set
$$ K=\{t^1x_1+\dots+t^m x_m:\;t^i\in \R_+,\;i=1,\dots,m\}$$
with $x_1,\,\dots,\,x_m$ linearly independent vectors, is called a \emph{simplicial cone}.
A simplicial cone is closed, pointed and generating.

The relation of the notion of simplicial cones to vector lattice theory relies
an the important result of Youdine (\cite{Youdine1939}):
 
\begin{theorem}[Youdine]\label{Yu}
$(\R^m,K)$ is a vector lattice with continuous lattice operations 
if and only if $K$ is a simplicial cone. 
\end{theorem}

(For this reason in vector lattice theory simplicial
cones sometimes are called Youdine cones as well.)

\begin{definition}\label{HVL}

The Hilbert space $H$ ordered by the order relation induced by the cone $K$
is called a \emph{Hilbert vector lattice} if (i) $K$ is a lattice cone, 
(ii) $\||x|\|=||x\|,\,\forall x\in H$, (iii) $0\leq x\leq y$ implies $\|x\|\leq\|y\|.$

\end{definition}

The cone $K$ is called \emph{self-dual}, if $K=K^*.$ If $K$
is self-dual, then it is a generating, pointed, closed convex cone.

In all that follows we shall suppose that $\R^m$ is endowed with a
Cartesian reference system with the standard unit vectors $e_1,\dots,e_m$.

The set
\[\R^m_+=\{x=(x^1,\dots,x^m)\in \R^m:\; x^i\geq 0,\;i=1,\dots,m\}\]
is called the \emph{nonnegative orthant} of the above introduced Cartesian
reference system. A direct verification shows that $\R^m_+$ is a
self-dual cone. The order relation $\leq_{\R^m_+}$ induced by $\R^m_+$ is called \emph{coordinate-wise ordering}.

The closed convex cone $K$ is called \emph{subdual} if $K\subseteq K^*$. 

The closed convex cone $K^\circ=-K^*$ is called the polar of $K$. We have 
$K^\circ=\{x\in H:\langle x,y\rangle\leq0,\,\forall y\in K\}$ and if $K$ is closed
$(K^\circ)^\circ=K$ (Farkas lemma). Therefore, the
closed convex cones $K$ and $L$ are called \emph{mutually polar} if $K=L^\circ$ (or equivalently $K^\circ=L$).


\section{New definitions}

By the \emph{image} $\im \varphi$ of a mapping $\varphi$ we mean the set of points $\varphi(x)$, where $x$ goes 
through all elements of the domain of definition of $\varphi$. 

\begin{definition}\label{vassnor}
Suppose that $K\subset X$ is a cone defining the order relation $\leq$. 
The mapping $Q:X\rightarrow K$ is an
\emph{asymmetric vector norm} (shortly aVn) if
	\begin{enumerate}
		\item $Q$ is \emph{retraction onto} $K$, i.e. $Q^2=Q\circ Q=Q$
			and $\im Q=K$;  
		\item $Q$ is \emph{positively homogeneous}, i,e,
		$Q(tx)=tQx\,\,\forall t\in \R_+,\,\forall x\in X$;
		\item $Q$ is \emph{subadditive}, that is $Q(x+y)\leq Qx+Qy,\,\,\forall \,x,\,y\in X$;
		\item $Qx=Q(-x)=0 \Rightarrow x=0$.
	\end{enumerate}
	The cone $K$ will be called \emph{the cone range} of $Q$.

If the ambient space $X$ is a topological vector space, we 
\emph{always suppose that $K$ is closed
and $Q$ is continuous.}
\end{definition}

{\bf Convention: To simplify the exposition 
we say in what follows about the operator $R$ that it is an aVn, or it is isotone  or subadditive, if the involved
order relation is with respect to the cone range $R(X)$ of $R$.}

\vspace{3mm}

From the definition of an aVn we have some direct conseguences:

\begin{remark}\label{alap}
\begin{enumerate}

\item If $Q:X\to K$ is an aVn then
the set $K^Q=\{x\in X:Qx=0\}$ is a pointed cone.
If the ambient space is a topological vector space 
then $K^Q$ is also closed.
	 
\item The vector space structure of $X$ induces in a
natural way a vector space structure in the set
of automorphisms $\mathcal A(X)$ of $X$. The presence of an
ordering in $X$ induces an ordering in the vector space $\mathcal A(X)$ too.
The set of aVn-s with the same cone range is convex.

\end{enumerate}
\end{remark}

\begin{example}\label{ska}
Suppose that $q$ is an asymmetric norm on $X$, i.e., a functional satisfying the conditions
in Definition \ref{assnor}. Suppose $x\in X$ is an element with $q(x)=1$. Let be $K=\{tx: t\in \R_+\}$.
Then $K$ is obviously a cone. Denote $\leq$ the order relation defined by $K$.

Define $Q:X\rightarrow K$ by $Q(y)= q(y)x$. Then $Q(x)=x$ and for any $y\in X$ we have
$Q^2(y)=Q(Q(y))=Q(q(y)x)=q(y)Q(x)=q(y)x=Q(y)$ that is $Q$ is a retraction onto $K$ and item 1
of the Definition \ref{vassnor} holds.

2. $Q(ty)=q(ty)x=tq(y)x=tQ(y).$

3. $Q(y+z)=q(y+z)x \leq (q(y)+q(z))x = Q(y)+Q(z).$

4.$Q(y)=Q(-y)=0 \Rightarrow q(y)x=q(-y)x=0.$ Hence $q(q(y)x)=q(q(-y)x)=q(y)q(x)=q(-y)q(x)=0 \Rightarrow q(y)=q(-y)=0$
and thus $y=0$.

Conversely, assume that $Q:X\rightarrow K\subset X$ is a vectorial asymmetric norm with $K$ an one dimensional cone. Let
$x$ be a nonzero element in $K$. Then $Q$ can be represented as $Q(y)= q(y)x,$ where $q:X\rightarrow \R_+$.
A direct checking shows that $q$ is an asymmetric norm.

\end{example}

This example shows that the notion of aVn-s is a reasonable extension of that of asymmetric norms
and that the requirement on an aVn of being a retract is natural. 

\begin{definition}
 The aVn-s with $K$ of dimension one will be called \emph{range one} aVn-s.
\end{definition}

\begin{example}\label{latt}

Let $(X,\leq)$ be a Riesz space with $K$ its positive cone. Then
using the properties of the positive part operator we can see that
$L=^+$ is an aVn.

Indeed, $^+: X \rightarrow K$ is obviously a retraction onto $K$.

$(tx)^+=tx^+\,\forall t\in \R_+,\, \forall x\in X$ and $(x+y)^+\leq x^++y^+$
from the definition of the supremum. 

If $x^+=(-x)^+ =0$, then $x,\,-x\in K$ hence $x=0$ since $K$ is pointed.

Besides the four axioms of aVn in Definition \ref{vassnor}, the positive part operator satisfies:

5.$^+$ is \emph{isotone}, i.e. $x\leq y \Rightarrow x^+\leq y^+$;

6. $I-^+= -^-$, where $I$ is the identity mapping,
it is an aVn with respect to the cone $-K$.

7. $(I-^+)^+=0$

\end{example}

\begin{definition}\label{lattaVn}
The aVn who is 
the positive part operator  is called \emph{lattice} aVn, and
to emphasize this and unify the notations we shall denote its positive part operator by $L$. 

Thus a lattice aVn $L$ satisfies the conditions of Definition \ref{vassnor} 
and the conditions 5, 6 and 7 in the above example.
\end{definition}


\section{Proper asymmetric vector norms}

We shall be in keeping with the convention that for any operators $R,\,S,\,U$
on a vector space $(R+S)U$ means $RU+SU$. But in general $U(R+S)$ is not $UR+US$.
The equality holds for instance if $U$ is linear.

\begin{definition}\label{proper} 
The aVn $Q:X \to K$ is said \emph{proper}, if
\begin{equation}\label{pp}
 Q(I-Q)=0.
\end{equation} 

\end{definition}

For an arbitrary cone $K\subset X$ and an arbitrary retraction $R:X\to K$ we
define $K^R=\{x\in X: Rx=0\}$. We say $K^R$  to be the $R$-\emph{polar} of $K$.

(We have adopted this notation and terminology to be in keeping with that used in the
literature for similar notions. Note that $K^R$ depends on $R$, and only implicitly
on $K$ by the fact that $K$ is the $R$-s cone range.)

\begin{lemma}\label{qpro}
The range one aVn defined by
$Qy=q(y)x$ with $q$ an asymmetric norm with $q(x)=1$,
is proper if and only if

\begin{equation}\label{felt}
q(y-q(y)x)=0,\,\forall y\in X.
\end{equation}

\end{lemma}

\begin{remark}\label{noproper}

\begin{enumerate}

\item As far as $Q$ is a retraction, $I-Q$ satisfies the relation (\ref{pp}).
Indeed
$$(I-Q)(I-(I-Q))=(I-Q)Q=Q-Q^2=0.$$
Hence, if $I-Q$ is an aVn, then it is proper.

\item From the condition 7 on the lattice aVn it follows that $L$ is proper.
It follows from 6 that $I-L$ is also a proper lattice aVn
and by 5 that they are also both isotone.

\item If we ignore condition 4 in the definition of a lattice aVn,
we can furnish simple examples of retractions satisfying conditions 1, 2, 3, 5, 6, 7
as the following example shows.

\end{enumerate}

\end{remark}

\begin{example}\label{primitiv}

Consider the two dimensional Euclidean space $R^2$ endowed with a Cartesian
reference system. Let
$K=(0,\R_+)$ and let $S:\R^2\to K$ defined by $S(u,v)=(0,v^+)$
Then a simple checking shows that $S$ satiisfies
the conditions 1, 2, 3, 5, 6, 7 in the definition of the lattice aVn.
\end{example}

\begin{proposition}\label{stand}

If $Q:X\to K$ is a aVn, then the following conditions are equivalent:

\begin{enumerate}

\item $Q$ is proper;

\item  $K^Q=\{x\in X: Q(x)=0\}= (I-Q)X$;

\item  $I-Q$ is a retraction onto the cone $K^Q$;

\end{enumerate}

In the above conditions the following obvious identities make sense:

 $(K^Q)^{I-Q} = \{x\in X: (I-Q)x=0\} = K$.

\end{proposition}

\begin{proof}

$1\Leftrightarrow 2.$

$(I-Q)X\subset K^Q$ by the definition of proper aVn-s.

If $x\in K^Q$, then $Qx=0$ and $x=(I-Q)x$, hence $x\in (I-Q)X$ and $K^Q\subset (I-Q)X.$

If 2 holds, then for an arbitrary $x\in X, (I-Q)x\in K^Q,$ hence $Q(I-Q)x=0.$

$1 \Leftrightarrow 3$

The equivalence is immediate from the identity $(I-Q)^2=(I-Q)(I-Q)=I-Q-Q(I-Q).$

\end{proof}

\begin{proposition}\label{prop1}

If $Q$ is a proper aVn with the cone range $K$ and $-K\subset K^Q$, then
$I-Q$ is a proper aVn too.

\end{proposition}

\begin{proof} 

In view of Remark \ref{noproper} item 1, it is sufficient to prove
that $I-Q$ is an aVn.

$I-Q$ is a retraction onto $K^Q=(I-Q)X$ by Proposition \ref{stand}.

$I-Q$ is obviously positive homogeneous.

For arbitrary $x,\,y\in X$ it follows
$$(I-Q)x + (I-Q)y -(I-Q)(x+y) = -Qx -Qy +Q(x+y) \in -K\subset K^Q.$$
Thus $I-Q$ is subadditive.

The condition $(I-Q)x = (I-Q)(-x)=0$ implies $Qx+Q(-x)=0$, hence $Qx=Q(-x)=0.$
Thus $x=0$

\end{proof}

\begin{proposition}\label{prop2}

If $Q$ and $R$ are two proper aVn-s with the cone range $K$,
$Q\leq R$ and $K^Q\not=K^R$ then no aVn of form
$$S=\lambda Q + \mu R,\, \lambda,\,\mu \in (0,1),\, \lambda + \mu =1$$
can be proper.

\end{proposition}

\begin{proof}

Since $Q\leq R$, it follows that 
\begin{equation}\label{neg}
K^R \subset K^Q,\, K^R\not=K^Q.
\end{equation}

We have $K^S=\{x\in X:\,(\lambda Q + \mu R)x=0\}.$ Since $Qx,\,Rx \in K$,
and $Qx\leq Rx,$ we have that $(\lambda Q + \mu R)x=0$ if and only if $Rx=0$.
Hence
\begin{equation}\label{neg1}
K^S=\{x\in X:\,(\lambda Q + \mu R)x=0\}=\{x\in X:\,Rx=0\}=K^R.
\end{equation}

Further, item 2 of Proposition \ref{stand} implies
\begin{equation}\label{neg2}
(I-S)X= \lambda (I-Q)X+ \mu(I-R)X = \lambda K^Q +\mu K^R=K^Q+K^R=K^Q
\end{equation}
according (\ref{neg}).

The relations (\ref{neg1}), (\ref{neg2}) together with (\ref{neg}) show
that
$$K^S=K^R\not= K^Q= (I-S)X,$$
Hence, by item 2 of Proposition \ref{stand}, $S$ is not proper.

\end{proof}

\begin{example}\label{exx}

For providing an example of a proper, one range aVn we consider
$\R^m$ endowed with a cartesian reference system.
For $y=(y_1,...,y_m)^T$ we put $q(y) = \mmx \{y_i^+:\,i=1,...,m\}$

Then $q$ is an asymmetric norm and $Qy=q(y)x$ 
with $x=(1,...,1)^T$, is an one range aVn.

The range cone of $Q$ is $K=Q(\R^m)= \R_+x$ and $K^Q=-\R^m_+$.

To verify condition (\ref{felt})
we observe that $y-q(y)x\in -\R^m_+$ and one has $q(y-q(y)x)=0,\,\forall y\in \R^m.$
Hence by Lemma \ref{qpro} , Q is proper.

Since $Q$ is proper, and $-K\subset -\R^m_+=K^Q$ it follows by Proposition \ref{prop1}
that $I-Q$ is a proper aVn too.

Observe that $Q$ is isotone, $I-Q$ is not isotone.

Indeed, $u\leq_K v$ is equivalent with $v=u+tx$ with some $t\in \R_+$.
Hence
$$Q(u+tx)=q(u+tx)x=\mmx\{(u_i+tx_i)^+\}x=(\mmx\{u_i^+\}+t)x =Qu+tx.$$
Thus $Q(u+tx)-Qu=tx\in K$ and the isotonicity of $Q$ follows.

If $u=x$ and $v=(2,1,...,1)^T$, then $u-v\in K^Q$, $(I-Q)u=0$,
$(I-Q)v=(0,-1,...,-1)^T$ and hence $(I-Q)u-(I-Q)v\notin K^Q,$
that is $I-Q$ is not isotone.

Observe that $x=(1,...,1)^T$ is the only element with $q(x)=1$ for which
$Qy=q(y)x$ is a proper aVn.

Indeed, if $u\in \R^m$ has the property that $q(u)=1$ and $u\not=x$, then $u$
has at least one coordinate $u_j<1$.  Thus
$x-q(x)u= x-u$ and $x_j-u_j=1-u_j >0$, relation showing that
$q(x-q(x)u)>0$, contradicting (\ref{felt}).

\end{example}

The condition on a range one aVn of being proper is rather restrictive.
Simple, standard one range aVn fails to have this property
as the following example shows.

\begin{example}\label{nopro}

Let $(H,\leq)$ be a a Hilbert vector lattice with the positive cone $K$.
Since the norm in this space is monotone, it is easy to show that
$ q(y)=\|y^+\|$ is an asymmetric norm and if $\ddm X\geq 2$ no one range
$Qy = q(y)x$ with $x=x^+$, $\|x\|=1$ can be proper.

Indeed, from an easy geometrical reasoning there exists $y\in K$
with $\|y\|=1$ such that $(y-x)^+\not=0$. Hence
$q(y-q(y))x=(\|(y-\|y^+\|x)^+\|=\|(y-x)^+\|\not=0$. which 
according Lemma \ref{qpro} shows that $Q$ cannot be proper.

\end{example}


\section{Construction of proper range one aVn-s}
  
\begin{proposition}\label{cons}

Suppose that $g: X\to \R_+$ is an asymmetric norm.
Define
\begin{equation*}
q:\R\times X\to \R_+ \;\;\textrm{by}\;\;q(t,x)=(t+g(x))^+.
\end{equation*}
Then $q$ is an asymmetric norm satisfying the relation
\begin{equation}\label{cons2}
q((t,x)-q(t,x)(1,0))=0,\;\;\forall \;(t,x), 
\end{equation}
and hence
\begin{equation}\label{cons3}
Q:\R\times X \to \R\times X,\;\; Q((t,x))=q((t,x))(1,0)
\end{equation}
is a proper range one aVn.
\end{proposition}

\begin{proof}

The functional $q$ is obviously positively homogeneous.
To prove its subadditivity, we must verify the relation
\begin{equation}\label{cons4}
q((t_1,x_1)+(t_2,x_2))=(t_1+t_2+g(x_1+x_2))^+\leq (t_1+g(x_1))^+ + (t_2+g(x_2))^+= q((t_1,x_1))+ q((t_2,x_2)).
\end{equation}
If all the involved sums in round brackets  are non-negative, the relation (\ref{cons4}) results from
the subadditivity of $g$.

Suppose that $t_1+g(x_1) \leq 0$. Then
$$t_1+t_2 +g(x_1+x_2) \leq -g(x_1) +t_2 +g(x_1+x_2)\leq -g(x_1) +t_2 +g(x_1) +g(x_2)= t_2 +g(x_2) $$
and the relation (\ref{cons4}) follows independently from the signs of the involved terms.

Suppose $q((t,x))=q(-(t,x))=0$, that is that $t+g(x)=-t +g(-x)=0.$ 
Then $g(x)+g(-x)=0$ and since both the terms are non-negative, they must be zero.
We have first that $x=0$ and then that $t=0$.

Thus $q$ is an asymmetric norm. To verify relation (\ref{cons2}) it is sufficient to
consider that $q((t,x))=(t+g(x))^+=1.$ Then it becomes 
$$q((t,x)-(1,0))=0,\;\;\textrm{with} \; q((t,x))=1.$$
But $q((t,x)-(1,0))=q((t-1,x))=(t-1+g(x))^+=0$, and it follows the relation (\ref{cons2}).

According to Lemma \ref{qpro}, the operator $Q$ defined at (\ref{cons3}) is
a proper range one aVn.

\end{proof}

\begin{corollary}\label{cons5}
Let $Q$ be the aVn constructed in Proposition \ref{cons}. Then
denoting $K=\R_+(1,0)$ we have
$$K^Q=\{(t,x):\,t+g(x)\leq 0\}$$
and since $-K\subset K^Q$, using Proposition \ref{prop1}, we conclude
that $(I-Q)((t,x))= (t,x)-q((t,x))(1,0)$ is a proper aVn with
the cone range $K^Q.$
\end{corollary} 


\section{Proper aVn in the Euclidean space with\\ proper cone range}

The proper cone in the Euclidean space is a pointed, closed convex cone with 
nonempty interior.

The subset $D\subset K$ is called \emph{the base of the cone} if for each
$x\in K\setminus \{0\}$ there exists a unique $\lambda >0$ such that
$\lambda x\in D$.

The following is a standard result in the convex geometry:

\begin{lemma}\label{proco}

If $K\subset \R^{m+1}$ is a proper cone, then, by choosing an appropriate system of
reference, we can find a closed, bounded, 
convex base $D$, of dimension $m$, such that a line through $0$ meeting
the interior of $K$,  is orthogonal to $D$ and the hyperplane through $0$
orthogonal to this line meets $K$ in the single point $0$.

\end{lemma}

\begin{theorem}\label{exavn}

If $K$ is a proper cone in $\R^{m+1},\;m \geq 2$, then there exists a proper
aVn $Q: \R^{m+1} \to \R^{m+1}$, whose cone range is $K$, and $K^Q$ is one dimensional.

\end{theorem}

\begin{proof}

From Lemma \ref{proco} it follows that with an appropriate choice of the 
reference system in $\R^{m+1}=\R \times \R^m$, we can find a closed, 
bounded base $D$ of $K$ of dimension $m$
orthogonal to the ray $(-\R_+,0) \subset K$. The hyperplane through $0$ 
orthogonal to this ray meets $K$ in the single point $0$, and this ray meets
$D$ in its relative interior point $(-1,0)$, where we have denoted an element in
$\R\times \R^m$ with $(t,x),\;t\in \R, \;x\in \R^m$. Hence if 
$(t,x)\in K\setminus \{0\}$, then $t<0$. The
$m$ dimensional hyperplane $H$ consisting 
of points of form $(-1,x),\; x\in R^m$ is orthogonal to $(-\R_+,0)$.

Identifying $H$ with an $m$-dimensional Euclidean space and denoting with $D'$ the set $D$
considered as subset of this space, we have the following picture: 
$D'$ is a bounded, closed convex set containing $0$ (identified with $(-1,0)$) of this space 
as interior point. Denote by $g$ the gauge of $D'$; then it will be an
asymmetric norm in an $m$-dimensional space (see e. g. \cite{Cobzas2013}, p. 165).

We shall show that
\begin{equation}\label{gauge}
K=\{(t,x): t+g(x)\leq 0\}.
\end{equation}

Take $(t,x)\in K\setminus \{0\}$. Then there exists the unique positive
$\lambda$ with $\lambda (t,x)\in D$. Thus $\lambda (t,x)=(-1,\lambda x).$
Hence $\lambda t=-1$ and $\lambda x\in D'$. Accordingly, $g(\lambda x)\leq 1$.
Putting $\lambda = -\frac{1}{t}$ in the last relation, we get 
$g(-\frac{1}{t} x) \leq 1$ and since $t$ is negative and $g$ 
is positive homogeneous, $g(x)\leq -t$, that is $t+g(x)\leq 0.$ Hence 
$(t,x)$ is in the second term of the equality (\ref{gauge}).

Conversely, if $t+g(x)\leq 0$, then $g(-\frac{1}{t} x)\leq 1$, hence
$-\frac{1}{t} x \in D'$. Thus $(-1,-\frac{1}{t} x)\in D$ and
$-t(-1,-\frac{1}{t}x) = (t,x)\in K,$ which completes the 
proof of the relation (\ref{gauge}).

Using now Proposition \ref{cons} and the Corollary \ref{cons5},
we conclude that
$q((t,x))=(t+g(x))^+$ is an asymmetric norm satisfying the relation (\ref{felt}),
hence $R=q((t,x))(1,0)$ is a proper aVn with the cone range $L=\R_+(1,0)$.
Since $-L\subset L^R=K$, it follows that $Q=I-R$ is a proper aVn with 
the cone range $K$ and $K^Q=L.$
\end{proof}


\section{Asymmetric vector norm having cone range and polar with non-empty interiors}

Let $H$ be a Hilbert space. The mapping $\zeta:H\to H$ is called \emph{sharp} if $\zeta(0)=0$ 
and $\im\zeta\cap\im(-\zeta)=\{0\}$. If $\zeta$ is sharp, then $t\zeta$ is also 
sharp, for all $t\in\mathbb R$. 
			
The following theorem of S. Z. N\'emeth ((\cite{Nemeth2010y} Theorem 2 and \cite{NemethNemeth2011x} Theorem 5.2
point 3) will be an important toll in our proofs:

\begin{theorem}\label{sanyi1}
Let $K$ be a closed, generating normal cone in the Hilbert space $H$. Then
$K$ is a lattice cone if and only if there exists a $\leq_K$-isotone continuous retraction 
$R:H\to K$  onto $K$ such that $I-R$ is sharp. 
\end{theorem}

\begin{theorem}\label{foo}
Let $K$  be a closed, generating normal cone in the separable
Hilbert space $H$.

For the proper aVn $Q$ with the cone range $K$
the following two assertions are equivalent:
\begin{enumerate}

\item $K$ is a lattice cone and $Q$ is its lattice aVn.

\item $K^Q=-K$.

\end{enumerate}
\end{theorem}

\begin{proof}

The implication $1\Rightarrow 2$ is obvious.

To show the implication $2 \Rightarrow 1$ we prove first the following assertion:

(a) \emph{Let $Q$ be a proper aVn with the cone range $K\subset H$
which is closed normal and generating. If $K\subset -K^Q$, then
$K$ is a lattice cone and $Q\leq L$ , where $L$ is the lattice aVn of $K$.}

Take $v\in K^Q$ and $x\in X$ arbitrarily. Then

$$Qx\leq Q(x-v)+Qv=Q(x-v).$$

This means that $Qx\leq Qy$ whenever $y-x \in -K^Q$. Since $K\subset -K^Q$
it follows that $Qx\leq Qy$ whenever $y-x\in K$ that is, whenever $x\leq y$.
Thus $Q$ is isotone. Since, by item 2 of Proposition \ref{stand}, $K^Q=(I-Q)X$ is pointed, it follows that
$I-Q$ is sharp, Theorem \ref{sanyi1} applies and  $K$ is a lattice cone.

Let $x\in H$ be arbitrary. Then $x\leq Lx$ and using the isotonicity of $Q$ we have
$Qx\leq Q(Lx)=Lx.$

(b) Suppose $x\leq Qx,\,\forall x\in H$. Then for an arbitrary $v\in H$ with
$x\leq v$ it holds $x\leq Qx\leq Qv$ and whenever $Qx \leq w$ for some $w\in K$
it follows that $x\leq Qx\leq Qw=w.$ Hence $Lx=x\vee 0 = Qx$.

(c) As $K=-K^Q$, it follows $x-Qx\in -K$. Thus $x\leq Qx,\,\forall x\in H$, the above
reason applies and $Q=L$.

\end{proof}

\begin{theorem}\label{foo1}

If $Q$ and $I-Q$ are both proper asymmetric vector norms in $H$ having closed generating normal cone ranges with non-empty
interiors, then they are lattice aVn-s.

\end{theorem}

\begin{proof}

Denote by $K$ and $K^Q$ the cone ranges of $Q$ and $I-Q$, respectively.

Since $\intr K^Q \not= \emptyset$, we see that
\begin{equation}\label{mf}
\{Qx+Qy-Q(x+y):\, x,y \in H\} = K.
\end{equation}

Indeed, from the subadditivity of $Q$ the left hand side term in (\ref{mf}) is
obviously subset of $K$

Let us take $x\in K$ and $y\in \intr K^Q$. For a sufficiently large $t>0$ we get 
$$\frac{1}{t}x+y\in K^Q \Rightarrow x+ty\in K^Q,$$
and hence
$$x=Qx+Q(ty)-Q(x+ty),$$
relation showing that $K$ is subset of the set in the left hand side of
(\ref{mf}), and thus the relation (\ref{mf}) follows.

Since $\intr (K^Q)^{I-Q}=\intr K \not=\emptyset$, the
similar relation for $I-Q$ will be
\begin{equation}\label{mf1}
\{(I-Q)x+(I-Q)y -(I-Q)(x+y):\,x,y\in H\}=K^Q.
\end{equation}

Relations (\ref{mf1}) and (\ref{mf}) together yield 
$$K^Q=-\{Qx+Qy-Q(x+y):\,x,y\in H\}=-K.$$

Using now Theorem \ref{foo}, the theorem follows.

\end{proof}


\section{Metric projections as asymmetric vector norms}

For a closed, pointed cone in the Hilbert space the distance to the cone is an
assymetric norm (see e. g. Proposition 2.3 in \cite{Conradie2015}).
 
This section answers the question of when
the metric projection onto the cone can be an aVn.

If $(H,\lang, \rang)$ is a real separable Hilbert space, then the projection, or the nearest point
operator onto a nonempty convex set is well defined.

Let $P:H\to K$ be the \emph{projection mapping} onto the closed  cone $K$, that is, the
mapping defined by  
$$Px=\argmin \{\|x-y\|:y\in K\}.$$ 

The following theorem is proved in \cite{Moreau1962}.

\begin{theorem} [Moreau]
	Let $H$ be a Hilbert space, $K,L\subset H$ two mutually polar closed 
	convex cones in $H$. Then, the following statements are equivalent:
	\begin{enumerate}
		\item $z=x+y,~x\in K,~y\in L$ and $\lang x,y\rang=0$,
		\item $x=P_K(z)$ and $y=P_L(z)$
	\end{enumerate}
\end{theorem}

Let us denote by $P$ the projection onto $K$. Since $K$ and $K^\circ$ are
mutually polar (Farkas lemma), we have from Moreau's theorem the identity

\begin{equation}\label{mor}
x=Px+(I-P)x \quad \textrm{with} \quad \lang Px, (I-P)x \rang =0,
\end{equation}
and the important consequence, that if $x=u+v$ with $u\in K$
and $v\in K^\circ$ and $\lang u,v\rang =0$, then we must have $u=Px$
and $v=(I-P)x.$

From (\ref{mor}) $Px=0$ if and only if $x\in (I-P)H=K^\circ= K^P$ and thus

\begin{equation}\label{ppro}
P(I-P)=0,
\end{equation}
relation what shows that the retraction $P$ (and its complement $I-P$)
satisfies condition (\ref{pp}) in the definition of a proper aVn-s.

Hence the question of when $P$ can be a proper aVn reduces to the 
characterization of the cone $K$ for which $P$ is subadditive. (The rest of conditions
in the definition of an aVn are trivially satisfied.) We have 
simply formulated characterizations, but whose verification imply
long and sophisticated proofs going back to those in \cite{IsacNemeth1986}. 

The problem of when the metric projection $P$ onto a cone is an aVn  and of when
$P$ and $I-P$ are both aVn-s is completely
settled by the following theorems which are
simple transliterations of our earlier results
using the terminology introduced in this note.

Gathering results from the Theorem 5.2 in \cite{NemethNemeth2012} and Theorem 3 in 
\cite{Nemeth2012d} we have:

\begin{theorem}\label{fooo}

If $(H,\lang, \rang)$ is a separable Hilbert space, $K\subset H$ is
a closed, normal generating cone, and $P$ is the projection onto $K$,
 then the following assertions are equivalent:
\begin{enumerate}
\item $P$ is a proper aVn; 
\item $K$ is a lattice cone and the operator $I-P$ is isotone with respect to the order engendered 
by $(I-P)H$;
\item If $\ddm H=m$ then conditions 1 and 2 are equivalent with
$$K= \{t^1x_1+...+t^mx_m:\,t^i\in R_+,\,i=1,...,m\},$$
where $x_1,...,x_m$ are linearly independent vectors in $H$ with
the property $\lang x_i,x_j\rang \leq 0$ whenever $i\not=j.$
\end{enumerate}

\end{theorem}

The cone satisfying the conditions of this theorem
is called \emph{coisotone}.
If $P$ is subadditive, that is, its cone range is coisotone, it
does not follow the same condition on $I-P$. This can happen only 
for a very restrictive condition on $K$.

Using Theorem 3 in \cite{Nemeth2003} and the above theorem we 
can conclude the following result:

\begin{theorem}\label{Hlatt}

For the separable Hilbert space ordered by the closed, generating normal cone $K\subset H$
and $P$ the projection onto $K$, the following assertions are equivalent:
\begin{enumerate}

\item  $P$ and $I-P$ are proper aVn-s;

\item $P$ is subadditive and isotone retraction;

\item $H$ is a Hilbert vector lattice;

\item $K$ is a selfdual lattice cone;

\item $K$ is a lattice cone and $P=L$ with $L$
the lattice aVn of $K$.

\item If $\ddm H=m$, then the above conditions are equivalent with

$$K= \{t^1x_1+...+t^mx_m:\,t^i\in R_+,\,i=1,...,m\},$$

where $x_1,...,x_m$ are linearly independent vectors in $H$ with
the property $\lang x_i,x_j\rang = 0$ whenever $i\not=j.$

\end{enumerate}

\end{theorem}

\begin{proof}

In the cited paper only the equivalence $2 \Leftrightarrow 3$ was proved.
Accepting it, we have to verify the rest of the equivalences.

$2\Rightarrow 4$.

Since $P$ is isotone, it follows that $K$ is subdual.

Indeed, for $x\in K$, $-x\leq_K 0$ and from the isotonicity of $P$ it holds
$0\leq P(-x)\leq P0=0.$ From (\ref{mor}) then $-x\in K^\circ =-K^*$,
and we have $K\subset K^*$.

From Theorem \ref{fooo}, $I-P$ is subadditive and  isotone too. 
 A similar argument for $K^* $ yields 
$K^* \subset (K^*)^*=K$, and item 4 is proved.

$4\Rightarrow 5$. 

Since $K=-K^\circ =-K^P$, we have from Theorem \ref{foo} that
$K$ is a lattice cone and $P=L.$

$5 \Rightarrow 1$. 

Since $P$ is a lattice aVn, it follows that $P$
and $I-P$ are both proper aVn-s.

$1\Rightarrow 2$. 

From Theorem \ref{fooo}, $P$ and $I-P$ are
booth subbadditive and isotone.

If $\ddm H=m$, then  a lattice cone $K$ must be simplicial by Theorem \ref{Yu}.
It is well known that the 
selfdual simplicial cones are isomorphic with $\R^m_+.$
This proves that in finite dimension we have the
equivalences of the statements at items 1--6.

\end{proof}

\begin{proposition}

Let $K$ be a not selfdual coisotone cone in  $H$.
Then the aVn $R=\lambda L+\mu P$ with $\lambda,\,\mu \in (0,1),\,\lambda +\mu =1$
cannot be proper.

\end{proposition}

\begin{proof}

Let $K$ be a coisiotone cone. We show that $L\leq P$.

From Theorem \ref{fooo}, $I-P$ is isotone with respect to the order engendered
by $K^\circ$. Since by Theorem \ref{fooo} $K^\circ$  is an isotone projection cone,  
one has $-K^\circ =K^* \subset K^{**}=K.$ Then $x-Px \in K^\circ \subset -K$
and thus $x\leq Px$. Since $L$ is isotone, it follows

$$Lx\leq LPx=Px,\,\forall x\in X.$$

\vspace{3mm}

 Since $K$ is not selfdual, by Theorem \ref{Hlatt} we must have that $L \leq P,\,K^L\not= K^P,$ and
we can apply Proposition \ref{prop2} to conclude the proof.

\end{proof}


\end{document}